\documentclass[preprint,12pt]{elsarticle}

\usepackage{amsmath,amssymb,amsthm}
\usepackage{mathtools}
\usepackage{microtype}
\usepackage{enumitem}
\usepackage{url}

\usepackage[colorlinks=true,
            linkcolor=blue!60!black,
            citecolor=blue!60!black,
            urlcolor=blue!60!black]{hyperref}

\journal{arXiv}

\theoremstyle{plain}
\newtheorem{theorem}{Theorem}[section]
\newtheorem{lemma}[theorem]{Lemma}
\newtheorem{proposition}[theorem]{Proposition}
\newtheorem{corollary}[theorem]{Corollary}
\newtheorem{conjecture}[theorem]{Conjecture}

\theoremstyle{definition}

\newtheorem{example}[theorem]{Example}

\theoremstyle{remark}
\newtheorem{remark}[theorem]{Remark}

\newcommand{\R}{\mathbb{R}}
\newcommand{\one}{\mathbf{1}}
\newcommand{\tr}{\operatorname{tr}}
\newcommand{\dedit}{d_{\mathrm{edit}}}

\begin{document}

\begin{frontmatter}

\title{The Bollob\'{a}s--Nikiforov inequality for complete multipartite
graphs, with remarks on the $K_4$-free case}

\author[inst1]{Piero Giacomelli\corref{cor1}}
\ead{pgiacome@gmail.com}
\cortext[cor1]{Corresponding author.}
\address[inst1]{IT Department, TENAX GROUP SPA, Verona, Italy}

\begin{abstract}
Bollob\'as and Nikiforov conjectured that every graph $G \neq K_n$ with
$m$ edges and clique number $\omega$ satisfies
$\lambda_1^2(G) + \lambda_2^2(G) \leq 2(1 - 1/\omega)m$, where
$\lambda_1(G) \geq \lambda_2(G)$ are the two largest adjacency
eigenvalues.
We give a short self-contained proof of the inequality for complete
multipartite graphs, based on a quadratic-form identity which shows
that $\lambda_2 \leq 0$, and we determine the equality cases in this
class.
We show that equality in the conjecture is attained by a wider family
than the Tur\'an graphs, namely by disjoint unions of two complete
bipartite graphs, or of two balanced complete $r$-partite graphs of
possibly different orders.
For $K_4$-free graphs we record the cases that follow from known results,
with corrected attributions, and we give examples showing that three
natural lines of attack fail: a Hoffman-type eigenvalue bound for graphs
with large independence number, a spectral stability statement for
connected graphs, and Zykov symmetrization in an arbitrary direction.
\end{abstract}


\end{frontmatter}

%% ============================================================
\section{Introduction}
\label{sec:intro}
%% ============================================================

All graphs are finite and simple. For a graph $G$ on $n$ vertices with
$m$ edges we write $\lambda_1(G) \geq \lambda_2(G) \geq \cdots \geq
\lambda_n(G)$ for the eigenvalues of its adjacency matrix $A(G)$ and
$\omega(G)$ for its clique number.

Nosal~\cite{Nosal1970} proved that $\lambda_1(G) \leq \sqrt{m}$ for every
triangle-free graph. Nikiforov~\cite{Nikiforov2002} extended this to all
graphs.

\begin{theorem}[Nikiforov~\cite{Nikiforov2002}]
  \label{thm:nikiforov}
  Every graph $G$ with $m$ edges and $\omega(G) = \omega \geq 2$ satisfies
  \begin{equation}
    \label{eq:nikiforov}
    \lambda_1^2(G) \;\leq\; 2\Bigl(1 - \frac{1}{\omega}\Bigr) m.
  \end{equation}
  Equality holds if and only if $G$, after removing isolated vertices,
  is a complete bipartite graph (when $\omega = 2$) or a complete regular
  $\omega$-partite graph (when $\omega \geq 3$).
\end{theorem}

Bollob\'as and Nikiforov~\cite{BollobasNikiforov2007} proposed the
following strengthening.

\begin{conjecture}[Bollob\'{a}s--Nikiforov~\cite{BollobasNikiforov2007}]
  \label{conj:BN}
  Every graph $G \neq K_n$ with $m$ edges and clique number $\omega$
  satisfies
  \[
    \lambda_1^2(G) + \lambda_2^2(G)
    \;\leq\;
    2\Bigl(1 - \frac{1}{\omega}\Bigr) m .
  \]
\end{conjecture}

The complete graph must be excluded: $K_n$ has $\lambda_1 = n-1$,
$\lambda_2 = -1$ and $2(1-1/n)\binom{n}{2} = (n-1)^2$, so
$\lambda_1^2 + \lambda_2^2 = (n-1)^2 + 1$ exceeds the bound.

\paragraph{Known cases.}
Lin, Ning and Wu~\cite{LinNingWu2021} proved the conjecture for
triangle-free graphs and characterized the extremal graphs.
For weakly perfect graphs, that is, graphs with $\chi(G) = \omega(G)$,
the conjecture follows from the inertial bound
$\chi(G) \geq 1 + s^+/s^-$ of Ando and Lin~\cite{AndoLin2015}, where
$s^{\pm}$ is the sum of the squares of the positive (resp.\ negative)
eigenvalues; see Section~\ref{ssec:known}.
Zhang~\cite{Zhang2024} proved the conjecture for regular graphs, with
equality exactly for balanced Tur\'an graphs and for disjoint unions of
two balanced Tur\'an graphs of the same order.
Kumar and Pragada~\cite{KumarPragada2024} proved a generalization due to
Elphick, Linz and Wocjan~\cite{ElphickLinzWocjan} for graphs with
$O(m^{3/2-\varepsilon})$ triangles, and Liu and Bu~\cite{LiuBu2025}
showed that the conjecture holds asymptotically almost surely for the
random graph $G(n,p)$.
Kannan, Kumar and Pragada~\cite{KannanKumarPragada} proposed an
edge-localized version and verified it for diamond-free graphs and random
graphs.
Most recently, Coutinho, Liu, Spier, Tang and Zhang~\cite{CoutinhoEtAl}
announced a proof of the conjecture for all graphs, in a stronger
weighted form, accompanied by a formalization in the Lean proof
assistant; at the time of writing this work has not been peer-reviewed.

\paragraph{Contents of this note.}
Our aim is modest and complementary.
In Section~\ref{sec:multipartite} we give a short self-contained proof of
the conjecture for complete multipartite graphs
(Theorem~\ref{thm:multipartite}).
The key step, $\lambda_2 \leq 0$, is classical (it follows from a theorem
of Smith~\cite{Smith1970}); we deduce it from an elementary
quadratic-form identity.
The inequality itself also follows from the weakly perfect case, since
complete $r$-partite graphs satisfy $\chi = \omega = r$.
We then exhibit a family of equality cases of Conjecture~\ref{conj:BN}
that is strictly larger than the family of Tur\'an graphs
(Proposition~\ref{prop:equality}).
In Section~\ref{sec:K4free} we collect what the known results give for
$K_4$-free graphs, and we record three examples showing that some natural
approaches to the $K_4$-free case cannot work.

The identity of Lemma~\ref{lem:form}, the spectral statement of
Lemma~\ref{lem:lambda2}, and the examples in
Section~\ref{ssec:cautionary} have been verified in the Lean~4 proof
assistant with the Mathlib library; the files are available from the
author.

%% ============================================================
\section{Preliminaries}
\label{sec:prelim}
%% ============================================================

For a real symmetric $n \times n$ matrix $M$ with eigenvalues
$\mu_1 \geq \cdots \geq \mu_n$ and $\mathbf{x} \neq \mathbf{0}$, write
$R_M(\mathbf{x}) = \mathbf{x}^\top M \mathbf{x} / \mathbf{x}^\top\mathbf{x}$.
We use the following standard facts; see, e.g.,~\cite{HornJohnson}.

\begin{theorem}[Courant--Fischer]
  \label{thm:CF}
  For $1 \leq k \leq n$,
  \[
    \mu_k = \min_{W} \max_{\mathbf{x} \in W \setminus \{\mathbf{0}\}}
    R_M(\mathbf{x}),
  \]
  where $W$ ranges over subspaces of $\R^n$ of dimension $n - k + 1$.
\end{theorem}

\begin{theorem}[Weyl]
  \label{thm:weyl}
  If $M$ and $E$ are real symmetric $n \times n$ matrices, then
  $\mu_k(M + E) \geq \mu_k(M) + \mu_n(E)$ for every $k$.
\end{theorem}

\begin{theorem}[Interlacing~\cite{HaemersInterlacing}]
  \label{thm:interlace}
  If $H$ is an induced subgraph of $G$ on $k$ vertices, then
  $\lambda_{n-k+i}(G) \leq \lambda_i(H) \leq \lambda_i(G)$ for
  $i = 1,\ldots,k$.
\end{theorem}

Since $A(G)$ has zero diagonal,
\begin{equation}
  \label{eq:trace}
  \sum_{i=1}^n \lambda_i = 0,
  \qquad
  \sum_{i=1}^n \lambda_i^2 = \tr A(G)^2 = 2m.
\end{equation}
In particular $\lambda_1(G) > 0$ whenever $m > 0$. The spectrum of a
disjoint union is the union (with multiplicities) of the spectra of its
components.

\begin{lemma}
  \label{lem:noncomplete}
  If $G$ has $n \geq 2$ vertices and is not complete, then
  $\lambda_2(G) \geq 0$.
\end{lemma}

\begin{proof}
Two non-adjacent vertices induce $2K_1$, whose eigenvalues are $0, 0$.
By Theorem~\ref{thm:interlace}, $\lambda_2(G) \geq \lambda_2(2K_1) = 0$.
\end{proof}

%% ============================================================
\section{Complete multipartite graphs}
\label{sec:multipartite}
%% ============================================================

Let $G = K_{n_1,\ldots,n_r}$ with parts $V_1,\ldots,V_r$, $r \geq 1$,
$n_i \geq 1$, and $n = n_1 + \cdots + n_r$. Two vertices are adjacent if
and only if they lie in different parts. For $\mathbf{x} \in \R^n$ put
$S_i(\mathbf{x}) = \sum_{v \in V_i} x_v$.

\begin{lemma}
  \label{lem:form}
  For every $\mathbf{x} \in \R^n$,
  \[
    \mathbf{x}^\top A(G)\,\mathbf{x}
    \;=\; \sum_{i \neq j} S_i(\mathbf{x}) S_j(\mathbf{x})
    \;=\; \Bigl(\sum_{v} x_v\Bigr)^{2} - \sum_{i=1}^r S_i(\mathbf{x})^2 .
  \]
\end{lemma}

\begin{proof}
$\mathbf{x}^\top A(G)\mathbf{x} = \sum_{u,v} A_{uv} x_u x_v$, and
$A_{uv} = 1$ exactly when $u \in V_i$, $v \in V_j$ with $i \neq j$.
Grouping the terms by parts gives $\sum_{i\neq j} S_i S_j$, and
$\bigl(\sum_i S_i\bigr)^2 = \sum_i S_i^2 + \sum_{i \neq j} S_i S_j$.
\end{proof}

\begin{lemma}
  \label{lem:lambda2}
  Let $G = K_{n_1,\ldots,n_r}$. Then $\lambda_2(G) \leq 0$, so $G$ has
  exactly one positive eigenvalue if $r \geq 2$. If moreover $n > r$, then
  $\lambda_2(G) = 0$.
\end{lemma}

\begin{proof}
Let $W = \{\mathbf{x} : \one^\top \mathbf{x} = 0\}$, a subspace of
dimension $n - 1$. For $\mathbf{x} \in W$, Lemma~\ref{lem:form} gives
$\mathbf{x}^\top A(G)\mathbf{x} = -\sum_i S_i(\mathbf{x})^2 \leq 0$.
By Theorem~\ref{thm:CF} with $k = 2$,
$\lambda_2(G) \leq \max_{\mathbf{x} \in W\setminus\{\mathbf{0}\}}
R_{A(G)}(\mathbf{x}) \leq 0$.
If $r \geq 2$ then $m > 0$, so $\lambda_1(G) > 0$ by~\eqref{eq:trace}.
If $n > r$, some part contains two vertices $u \neq w$; the rows of $u$
and $w$ in $A(G)$ coincide, so $\mathbf{e}_u - \mathbf{e}_w$ lies in the
kernel of $A(G)$. Thus $0$ is an eigenvalue, and since
$\lambda_2(G) \leq 0$ and only $\lambda_1(G)$ is positive, we get
$\lambda_2(G) = 0$.
\end{proof}

\begin{remark}
  Lemma~\ref{lem:lambda2} is a special case of a theorem of
  Smith~\cite{Smith1970}: a graph has exactly one positive eigenvalue if
  and only if its non-isolated vertices induce a complete multipartite
  graph. The proof above avoids computing the spectrum; for completeness
  we note that the nonzero eigenvalues of $K_{n_1,\ldots,n_r}$ other than
  $-p$ (with multiplicity one less than the number of parts of size $p$)
  are the roots of $\sum_{i} n_i/(\lambda + n_i) = 1$.
\end{remark}

\begin{theorem}
  \label{thm:multipartite}
  Let $G = K_{n_1,\ldots,n_r}$ with $r \geq 2$, $n \geq r+1$ and $m$ edges.
  Then
  \[
    \lambda_1^2(G) + \lambda_2^2(G) \;\leq\; 2\Bigl(1 - \frac{1}{r}\Bigr)m.
  \]
  If $r = 2$, equality holds for every complete bipartite graph
  $K_{n_1,n_2}$ with $n_1 + n_2 \geq 3$. If $r \geq 3$, equality holds if
  and only if $n_1 = \cdots = n_r$.
\end{theorem}

\begin{proof}
We have $\omega(G) = r$ and, by Lemma~\ref{lem:lambda2},
$\lambda_2(G) = 0$. Hence $\lambda_1^2 + \lambda_2^2 = \lambda_1^2$, and
the inequality is Theorem~\ref{thm:nikiforov}.
Since $G$ has no isolated vertices, Theorem~\ref{thm:nikiforov} gives
equality exactly when $r = 2$, or when $r \geq 3$ and $G$ is regular,
that is, all parts have the same size.
Directly: $\lambda_1(K_{a,b}) = \sqrt{ab}$ and $m = ab$; and for
$n_1 = \cdots = n_r = p$ the all-ones vector is an eigenvector with
eigenvalue $(r-1)p$, while $2(1-1/r)\binom{r}{2}p^2 = (r-1)^2p^2$.
\end{proof}

The equality cases of Conjecture~\ref{conj:BN} are not limited to Tur\'an
graphs and complete bipartite graphs.

\begin{proposition}
  \label{prop:equality}
  Fix $r \geq 2$ and call a graph \emph{$r$-extremal} if it is a complete
  bipartite graph $K_{a,b}$ $(a,b \geq 1)$ when $r = 2$, or a balanced
  complete $r$-partite graph $K_{p,\ldots,p}$ $(p \geq 1)$ when $r \geq 3$.
  Let $G = H_1 \cup H_2 \cup sK_1$ $(s \geq 0)$ be a disjoint union, where
  $H_1$ is $r$-extremal and $H_2$ is either empty or $r$-extremal.
  If $G$ is not complete, then $\omega(G) = r$ and
  \[
    \lambda_1^2(G) + \lambda_2^2(G) = 2\Bigl(1 - \frac{1}{r}\Bigr)m .
  \]
\end{proposition}

\begin{proof}
Clearly $\omega(G) = r$. For an $r$-extremal graph $H$ we have
$\lambda_1(H)^2 = 2(1-1/r)\,e(H)$ (proof of
Theorem~\ref{thm:multipartite}, which also covers $K_r$, the case
$p = 1$), and by Lemma~\ref{lem:lambda2} all other eigenvalues of $H$ are
at most $0$.
If $H_2$ is non-empty, the two positive eigenvalues of $G$ are
$\lambda_1(H_1)$ and $\lambda_1(H_2)$, so
$\lambda_1^2(G) + \lambda_2^2(G) = 2(1-1/r)\bigl(e(H_1) + e(H_2)\bigr)
= 2(1-1/r)m$.
If $H_2$ is empty, then $\lambda_1(G) = \lambda_1(H_1)$, and
$\lambda_2(G) = 0$: either $s \geq 1$, or $H_1$ has a part with at least
two vertices (otherwise $G = K_r$ is complete), and in both cases $0$ is
an eigenvalue while no other eigenvalue is positive.
\end{proof}

\begin{example}
  The graphs $2K_2$, $K_{2,3} \cup K_{1,4}$, $K_4 \cup K_4$,
  $K_4 \cup K_1$, $K_{2,2,2} \cup K_{3,3,3}$ and
  $K_{3,3,3} \cup 3K_1$ all attain equality in Conjecture~\ref{conj:BN}.
  The first is triangle-free but not complete bipartite, and the
  $\omega = 3$ examples are not Tur\'an graphs. On the other hand,
  equality fails for $K_{2,2,3} \cup K_{2,2,2}$ (unbalanced parts), for
  $K_{2,2,2} \cup K_{3,3}$ (components with different clique numbers) and
  for $3K_{2,2,2}$ (three components).
\end{example}

%% ============================================================
\section{The $K_4$-free case}
\label{sec:K4free}
%% ============================================================

For a $K_4$-free graph, $\omega(G) \leq 3$, and Conjecture~\ref{conj:BN}
reads $\lambda_1^2 + \lambda_2^2 \leq 4m/3$ whenever $G$ is not complete.

\subsection{Cases that follow from known results}
\label{ssec:known}

\begin{theorem}[Ando--Lin~\cite{AndoLin2015}]
  \label{thm:andolin}
  For every graph $G$ with at least one edge,
  $\chi(G) \geq 1 + s^+/s^-$, where $s^+$ and $s^-$ are the sums of the
  squares of the positive and of the negative eigenvalues of $A(G)$.
\end{theorem}

\begin{corollary}
  \label{cor:weakly-perfect}
  Conjecture~\ref{conj:BN} holds for every weakly perfect graph.
\end{corollary}

\begin{proof}
Let $G \neq K_n$ with $\chi(G) = \omega(G) = \omega$. If $m = 0$ there is
nothing to prove, so assume $m > 0$; then $n \geq 2$ and
$\lambda_2(G) \geq 0$ by Lemma~\ref{lem:noncomplete}, whence
$\lambda_1^2 + \lambda_2^2 \leq s^+$. By~\eqref{eq:trace},
$s^+ + s^- = 2m$, and Theorem~\ref{thm:andolin} gives
$s^+ \leq (\chi - 1)s^- = (\chi - 1)(2m - s^+)$, that is,
$s^+ \leq 2(1 - 1/\chi)m = 2(1-1/\omega)m$.
\end{proof}

\begin{theorem}
  \label{thm:3colorable}
  Let $G$ be a $K_4$-free graph with $\chi(G) \leq 3$ that is not complete.
  Then $\lambda_1^2(G) + \lambda_2^2(G) \leq 4m/3$.
\end{theorem}

\begin{proof}
If $\omega(G) = 3$ then $\chi(G) = \omega(G) = 3$ and
Corollary~\ref{cor:weakly-perfect} applies. If $\omega(G) \leq 2$, the
result of Lin, Ning and Wu~\cite{LinNingWu2021} (or
Corollary~\ref{cor:weakly-perfect} when $\chi(G) = 2$) gives
$\lambda_1^2 + \lambda_2^2 \leq m \leq 4m/3$; graphs with $m = 0$ are
trivial.
\end{proof}

\begin{remark}
  The hypothesis that $G$ is not complete cannot be dropped: $K_2$ is
  $K_4$-free with $\chi = 2$, but $\lambda_1^2 + \lambda_2^2 = 2 > 4m/3$.
\end{remark}

\begin{corollary}[Zhang~\cite{Zhang2024}]
  \label{cor:regular}
  If $G$ is a regular $K_4$-free graph that is not complete, then
  $\lambda_1^2(G) + \lambda_2^2(G) \leq 4m/3$.
\end{corollary}

\subsection{Three approaches that fail}
\label{ssec:cautionary}

We record three examples showing that some natural strategies for the
$K_4$-free case cannot succeed as stated.

\paragraph{Hoffman-type bounds.}
For a regular graph, Hoffman's ratio bound~\cite{Hoffman1970} states
$\alpha(G) \leq -n\lambda_n/(\lambda_1 - \lambda_n)$. It is tempting to
combine it with $\alpha(G) \geq n/3$ to obtain $|\lambda_n| \geq
\lambda_1/2$ and then use $\lambda_1^2 + \lambda_2^2 \leq 2m -
\lambda_n^2$. Both steps fail for irregular graphs.

\begin{example}
  \label{ex:hoffman}
  (a) The star $K_{1,3}$ has $\lambda_1 = \sqrt{3}$ and $\lambda_4 =
  -\sqrt{3}$, so the ratio bound would give $\alpha \leq 2$, while
  $\alpha(K_{1,3}) = 3$.

  (b) Let $G$ be the graph on $\{0,\ldots,5\}$ with edges $01$, $03$,
  $05$, $12$, $13$, $14$, $15$, $24$, $25$, $34$. Then $G$ is $K_4$-free with
  $\alpha(G) = 2 = n/3$, and its characteristic polynomial is
  \[
    (x^2 - 2x - 5)(x^2 + x - 1)^2 .
  \]
  Hence $\lambda_1 = 1 + \sqrt{6}$ and
  $\lambda_6 = -(1+\sqrt{5})/2$, so $|\lambda_6| < \lambda_1/2$.
\end{example}

Even when $|\lambda_n| \geq \lambda_1/2$ does hold, the resulting bound
$\lambda_1^2 + \lambda_2^2 \leq 2m - \lambda_1^2/4$ is at most $4m/3$
only if $\lambda_1^2 \geq 8m/3$, which contradicts~\eqref{eq:nikiforov}
for $\omega \leq 3$. So the approach is too weak in any case.

\paragraph{Spectral stability.}
One might hope that if $\lambda_1^2 + \lambda_2^2$ is close to $4m/3$,
then $G$ is close in edit distance to a complete tripartite graph.
Write $\dedit(G, \mathcal{T}_3)$ for the minimum number of edge
insertions and deletions turning $G$ into a complete tripartite graph on
the same vertex set (parts may be empty). The disjoint union
$K_{k,k,k} \cup kK_1$ satisfies $\lambda_1^2 + \lambda_2^2 = 4m/3$ but
is far from $\mathcal{T}_3$. The next proposition shows that requiring
connectivity does not help.

\begin{proposition}
  \label{prop:stability}
  For $k \geq 1$ let $G_k$ consist of two disjoint copies of $K_{k,k,k}$
  joined by a single edge. Then $G_k$ is connected and $K_4$-free, with
  $n = 6k$ and $m = 6k^2 + 1$, and
  \[
    \lambda_1^2(G_k) + \lambda_2^2(G_k) \;\geq\; 2(2k-1)^2,
    \qquad
    \dedit(G_k, \mathcal{T}_3) \;\geq\; 2k^2 - 1 .
  \]
  In particular $(\lambda_1^2 + \lambda_2^2)/m \to 4/3$, while
  $\dedit(G_k,\mathcal{T}_3) \geq n^2/18 - 1$.
\end{proposition}

\begin{proof}
$A(G_k) = A(2K_{k,k,k}) + E$, where $E$ is the adjacency matrix of a
single edge, with smallest eigenvalue $-1$. Since
$\lambda_1(2K_{k,k,k}) = \lambda_2(2K_{k,k,k}) = 2k$,
Theorem~\ref{thm:weyl} gives $\lambda_1(G_k), \lambda_2(G_k) \geq 2k - 1$.

Let $A$ and $B$ be the vertex sets of the two copies, and let $H$ be a
complete tripartite graph with parts $P_1, P_2, P_3$. Put $a_j =
|A \cap P_j|$, $b_j = |B \cap P_j|$ and $t = \min(\max_j a_j, \max_j
b_j)$; by symmetry assume $\max_j a_j = t$, attained at $j = 1$.
The number of edges of $H$ between $A$ and $B$ is
$\sum_{i \neq j} a_i b_j = 9k^2 - \sum_j a_j b_j \geq 9k^2 - 3kt$, and
all but at most one of them must be inserted.
If $t < k$, this already gives at least $9k^2 - 3k^2 - 1 \geq 2k^2 - 1$
edits. If $t \geq k$, the set $A \cap P_1$ consists of $t$ vertices of a
copy of $K_{k,k,k}$, say $s_1, s_2, s_3 \leq k$ of them in its three
parts, and it spans $\bigl(t^2 - \sum_i s_i^2\bigr)/2$ edges of $G_k$, all
of which must be deleted. By convexity, $\sum_i s_i^2 \leq k^2 + (t-k)^2$
if $t \leq 2k$ and $\sum_i s_i^2 \leq 2k^2 + (t-2k)^2$ if $t > 2k$; in
both cases at least $k(t-k)$ edges must be deleted. In total we need at least
$9k^2 - 3kt - 1 + k(t-k) = 8k^2 - 2kt - 1 \geq 2k^2 - 1$ edits, because
$t \leq 3k$.
\end{proof}

\paragraph{Zykov symmetrization.}
For non-adjacent vertices $u, v$ of $G$, let $Z(G;u,v)$ be the graph
obtained by replacing the neighbourhood of $u$ with that of $v$. This
operation does not increase the clique number. It is sometimes used as if
it never decreased $\lambda_1$, but the direction matters: if $G$ is
$K_{1,3}$ together with an isolated vertex $v$, and $u$ is the centre of
the star, then $\lambda_1(G) = \sqrt{3}$ while $Z(G;u,v)$ has no edges.
The correct statement is the following.

\begin{lemma}
  \label{lem:zykov}
  Let $\mathbf{x} \geq 0$ be a unit eigenvector of $A(G)$ for
  $\lambda_1(G)$, and let $u, v$ be non-adjacent vertices with
  $\sum_{w \in N(v)} x_w \geq \sum_{w \in N(u)} x_w$. Then
  $\lambda_1(Z(G;u,v)) \geq \lambda_1(G)$.
\end{lemma}

\begin{proof}
Since $u \notin N(v)$ and $v \notin N(u)$, only the terms involving $u$
change, and
$\mathbf{x}^\top A(Z(G;u,v))\mathbf{x} - \mathbf{x}^\top A(G)\mathbf{x}
= 2x_u\bigl(\sum_{w \in N(v)} x_w - \sum_{w \in N(u)} x_w\bigr) \geq 0$.
By the Rayleigh principle,
$\lambda_1(Z(G;u,v)) \geq \mathbf{x}^\top A(Z(G;u,v))\mathbf{x} \geq
\lambda_1(G)$.
\end{proof}

Repeated symmetrization in the right direction turns every $K_4$-free
graph into a complete multipartite graph with at most three parts
(together with isolated vertices). However, both $\lambda_1^2 +
\lambda_2^2$ and $m$ change along the way, and no monotonicity of their
ratio is known. So this does not reduce Conjecture~\ref{conj:BN} to
Theorem~\ref{thm:multipartite}.

%% ============================================================
\section{Concluding remarks}
\label{sec:remarks}
%% ============================================================

If the announced proof~\cite{CoutinhoEtAl} is confirmed,
Conjecture~\ref{conj:BN} is settled. Natural questions remain:
\begin{enumerate}
  \item Characterize all graphs attaining equality in
        Conjecture~\ref{conj:BN}. Proposition~\ref{prop:equality} gives
        a family that properly contains the Tur\'an graphs; the
        triangle-free case is settled in~\cite{LinNingWu2021}, and the
        regular case in~\cite{Zhang2024}.
  \item The generalization of Elphick, Linz and
        Wocjan~\cite{ElphickLinzWocjan}, which replaces
        $\lambda_1^2 + \lambda_2^2$ by the sum of the squares of the
        $\omega$ largest eigenvalues (when positive).
  \item The edge-localized version of Kannan, Kumar and
        Pragada~\cite{KannanKumarPragada}.
  \item A spectral stability theorem for Conjecture~\ref{conj:BN}.
        Proposition~\ref{prop:stability} shows that closeness of
        $\lambda_1^2 + \lambda_2^2$ to $4m/3$ does not force a connected
        $K_4$-free graph to be close to a complete tripartite graph, so
        any such statement must allow for the extremal unions of
        Proposition~\ref{prop:equality}.
\end{enumerate}

\section*{Acknowledgements}
The author thanks Dr. Shengtong Zhang for helpful discussions.

%% ============================================================
%%  References
%% ============================================================


\begin{thebibliography}{99}

\bibitem{AndoLin2015}
T.~Ando, M.~Lin,
Proof of a conjectured lower bound on the chromatic number of a graph,
\textit{Linear Algebra Appl.} \textbf{485} (2015) 480--484.

\bibitem{BollobasNikiforov2007}
B.~Bollob\'as, V.~Nikiforov,
Cliques and the spectral radius,
\textit{J. Combin. Theory Ser.~B} \textbf{97} (2007) 859--865.

\bibitem{CoutinhoEtAl}
G.~Coutinho, Y.~Liu, T.~Jung Spier, Q.~Tang, S.~Zhang,
A proof of the Bollob\'as--Nikiforov conjecture in Lean,
preliminary research note and formalization, 2026,
\url{https://github.com/shengtongz/BN}.

\bibitem{ElphickLinzWocjan}
C.~Elphick, W.~Linz, P.~Wocjan,
Two conjectured strengthenings of Tur\'an's theorem,
\textit{arXiv:2101.05229}.

\bibitem{HaemersInterlacing}
W.H.~Haemers,
Interlacing eigenvalues and graphs,
\textit{Linear Algebra Appl.} \textbf{226--228} (1995) 593--616.

\bibitem{Hoffman1970}
A.J.~Hoffman,
On eigenvalues and colorings of graphs,
in: \textit{Graph Theory and its Applications} (B.~Harris, Ed.),
Academic Press, 1970, pp.~79--91.

\bibitem{HornJohnson}
R.A.~Horn, C.R.~Johnson,
\textit{Matrix Analysis}, 2nd ed., Cambridge University Press, 2013.

\bibitem{KannanKumarPragada}
M.R.~Kannan, H.~Kumar, S.~Pragada,
Localization of spectral Tur\'an-type theorems,
\textit{arXiv:2512.01409}, 2025.

\bibitem{KumarPragada2024}
H.~Kumar, S.~Pragada,
Bollob\'as--Nikiforov conjecture for graphs with not so many triangles,
\textit{arXiv:2407.19341}, 2024.

\bibitem{LinNingWu2021}
H.~Lin, B.~Ning, B.~Wu,
Eigenvalues and triangles in graphs,
\textit{Combin. Probab. Comput.} \textbf{30} (2021) 258--270.

\bibitem{LiuBu2025}
C.~Liu, C.~Bu,
Bollob\'as--Nikiforov conjecture holds asymptotically almost surely,
\textit{arXiv:2501.07137}, 2025.

\bibitem{Nikiforov2002}
V.~Nikiforov,
Some inequalities for the largest eigenvalue of a graph,
\textit{Combin. Probab. Comput.} \textbf{11} (2002) 179--189.

\bibitem{Nosal1970}
E.~Nosal,
\textit{Eigenvalues of Graphs},
Master's Thesis, University of Calgary, 1970.

\bibitem{Smith1970}
J.H.~Smith,
Some properties of the spectrum of a graph,
in: \textit{Combinatorial Structures and their Applications}
(Proc. Calgary Internat. Conf., 1969), Gordon and Breach, New York, 1970,
pp.~403--406.

\bibitem{Zhang2024}
S.~Zhang,
On the first two eigenvalues of regular graphs,
\textit{Linear Algebra Appl.} \textbf{686} (2024) 102--110.

\end{thebibliography}
\end{document}